\documentclass[11pt,a4paper,twoside]{article}
\usepackage{bm,amsthm, amsfonts,amsmath}

\newtheorem{theorem}{Theorem}[section]
\newtheorem{corollary}{Corollary}[section]
\newtheorem{definition}{Definition}[section]
\newtheorem{example}{Example}[section]
\newtheorem{lemma}{Lemma}[section]
\newtheorem{proposition}{Proposition}[section]
\newtheorem{remark}{Remark}[section]

\def\Ric{\operatorname{Ric}}
\def\tr{{\rm\,trace\,}}
\def\eq{\hspace*{-2.1mm}&=&\hspace*{-2.1mm}}

\newcommand\Div{\operatorname{div}}
\def\vol{\operatorname{vol}}

\author{Vladimir Rovenski\footnote{Department of Mathematics, University of Haifa,
 Israel
       \newline e-mail: {\tt vrovenski@univ.haifa.ac.il}
       } }

\title{Einstein-type metrics and generalized Ricci solitons \\ on weak $f$-K-contact manifolds}

\begin{document}

\date{}

\maketitle

\begin{abstract}
A weak metric $f$-structure $(f,Q,\xi_i,\eta^i,g)\ (i=1,\ldots,s)$, generalizes the metric $f$-structure on a smooth manifold,
i.e., the complex structure on the contact distribution is replaced with a nonsingular skew-symmetric tensor.
We study geometry of a weak $f$-{K}-contact structure, which is a weak $f$-contact structure, whose characteristic vector fields are Killing.
We show that $\ker f$ of a weak $f$-contact manifold defines a $\mathfrak{g}$-foliation with an abelian Lie algebra.
Then we characterize weak $f$-{K}-contact manifolds among all weak metric $f$-manifolds by the pro\-perty known for $f$-{K}-contact manifolds, and find
when a Riemannian manifold endowed with a set of orthonormal Killing vector fields is a weak $f$-{K}-contact manifold.
We show that for $s>1$, an Einstein weak $f$-{K}-contact manifold is Ricci flat, then find sufficient conditions for a weak $f$-{K}-contact manifold with parallel Ricci tensor or with a generalized gradient Ricci soliton structure to be Ricci flat or a quasi Einstein~manifold.
We prove positive definiteness of the Jacobi operators in the characteristic directions and use this to deform a weak $f$-K-contact structure
to an $f$-K-contact structure.
We define an $\eta$-Ricci soliton and $\eta$-Einstein structures on a weak metric $f$-manifold
(which for $s=1$, give the well-known structures on contact metric manifolds)
and find sufficient conditions for a compact weak $f$-{\rm K}-contact manifold with an $\eta$-Ricci soliton structure of constant scalar curvature to be~$\eta$-Einstein.


\end{abstract}


\section{Introduction}
\label{sec:00}

Contact geometry is of growing interest due to its important role in mechanics in explaining physical phenomena.
In addition, many recent articles have been motivated by the question: how interesting Ricci solitons
-- self-similar solutions of the Ricci flow equation, $\partial g(t)/\partial t = -2\,\Ric_{g(t)}$ --
can be for geometry of contact metric manifolds.
Some of them consider conditions when a~contact metric manifold $M(g,f,\xi,\eta)$,
equipped with a Ricci-type soliton structure carries a canonical
(e.g., constant curvature or Einstein-type) metric, e.g.,~\cite{G-D-2020,G-2023,CZB-2022,N-R-2016}.
%
The~Ricci soliton equation
\[
 (1/2)\,\pounds_V\,g + \Ric = \lambda\,g,
\]
where $V$ is a vector field and $\lambda\in\mathbb{R}$,
generalizes Einstein metrics, $\Ric = \lambda\,g$. In the case that $V$ is the gradient of a potential function $\sigma$ one has a gradient Ricci~soliton equation ${\rm Hess}_{\,\sigma} + \Ric = \lambda\,g$. Almost Ricci solitons (when $\lambda\in C^\infty(M)$, see \cite{PRRS-2011}) are a generalization of quasi-Einstein manifolds (see \cite{Cao-2009}) and Ricci solitons.
The following {generalized Ricci soliton} equation was studied in \cite{N-R-2016}:
\begin{equation}\label{E-g-r-e}
 (1/2)\,\pounds_{V}\,g -c_2\Ric = \lambda\,g -c_1 V^\flat\otimes V^\flat ,
\end{equation}
where~$c_1,c_2,\lambda\in\mathbb{R}$ and $V^\flat$ is the 1-form dual to the vector field $V$.
If $V=\nabla\sigma$ in \eqref{E-g-r-e}, then
using the definition ${\rm Hess}_{\,\sigma}=\frac12\,\pounds_{\nabla \sigma}\,g$,
we get the \textit{generalized gradient Ricci soliton} equation
\begin{equation}\label{E-gg-r-e}
 {\rm Hess}_{\,\sigma} - c_2\Ric = \lambda\,g -c_1 d\sigma\otimes d\sigma\,,
\end{equation}
and the tensor product notation $(d\sigma\otimes d\sigma)(Y,Z)=d\sigma(Y)\,d\sigma(Z)$ is used.
Each equation above gives a generalization of Einstein metric manifolds.
%
For different values of $c_1,c_2$ and $\lambda$, equation \eqref{E-g-r-e} generalizes
Killing equation ($c_1 = c_2 = \lambda = 0$), equation for homotheties ($c_1 = c_2 = 0$), Ricci soliton equation
($c_1 = 0,\,c_2 = -1$), vacuum near-horizon geometry equation ($c_1 = 1,\ c_2 = 1/2$), e.g., \cite{G-D-2020}.
Catino et al. \cite{CMMR-2017} also introduced the notion of an Einstein-type structure on a manifold,
\begin{equation*}
 \beta\,{\cal L}_V\,g + \alpha\Ric = (\rho\,{\rm Scal} +\lambda)\,g +\mu\,V^\flat\otimes V^\flat
\end{equation*}
(with $\alpha,\beta,\mu,\rho\in\mathbb{R}$)
unifying various cases studied recently, such as gradient Ricci solitons, and quasi-Einstein manifolds.
There are no Einstein metrics on some compact manifolds, which motivated the study of genera\-lizations of such metrics.
A~contact metric manifold is called $\eta$-Einstein,~if
\[
 \Ric = \lambda\,g +\nu\,\eta\otimes\eta,\quad\lambda,\nu\in\mathbb{R} .
\]
Under certain conditions, a Sasaki-Ricci flow (an odd-dimensional counterpart of the K\"{a}hler-Ricci flow) on a compact manifold converges to an $\eta$-Einstein~metric on a contact metric manifold, see~\cite{SW-2010}. The notion of $\eta$-Ricci soliton on a contact metric manifold, see \cite{CJ-2009},
generalizes Ricci solitons and $\eta$-Einstein~metrics:
 $(1/2)\,\pounds_V\,g + \Ric = \lambda\,g +\nu\,\eta\otimes\eta,\quad\nu\in\mathbb{R}$.

A metric $f$-structure on a $(2n+s)$-dimensional smooth manifold is a higher dimensional analog of a contact structure,
defined by a (1,1)-tensor $f$ of constant rank $2n$, which satisfies $f^3+f=0$, and orthonormal vector fields $\{\xi_i\}_{1\le i\le s}$
spanning $\ker f$ -- the $2n$-dimensional characteristic distribution, see \cite{b1970,BP-2016,fip,tpw,yan,YK-1985}.
Foliations of simple extrinsic geometry, i.e., vani\-shing second fundamental form of the leaves,
appear on manifolds with degenerate differential forms and curvature-like tensors, e.g.~\cite{Rov-Wa-2021}.
The following problem was posed in \cite{fip}: {find suitable structures on manifolds, which lead to totally geodesic foliations}.
The~distribution $\ker f$ of an $f$-contact structure defines a totally geodesic foliation.
An~interesting case occurs when $\ker f$ is defined by a homomorphism of a $s$-dimensional Lie algebra $\mathfrak{g}$ to the Lie algebra of all vector fields on $M$, i.e., $M$ admits a $\mathfrak{g}$-foliation, see~\cite{AM-1995}.
In the presence of a compatible metric, such $\mathfrak{g}$-foliations are spanned by Killing vector fields, e.g., \cite{ca-toh}.
For a 1-dimensional Lie algebra, a ${\mathfrak{g}}$-foliation is gene\-rated by a nonvanishing vector~field,
and we get contact metric manifolds as well as $K$-contact and Sasakian ones, see~\cite{blair2010}.

An~$f$-{K}-contact structure, i.e. an $f$-contact structure, whose characteristic vector fields are unit Killing vector fields,
see \cite{Goertsches-2}, can be regarded as intermediate between a metric $f$-structure and $S$-structure
(the Sasaki structure when $s=1$).
The influence of constant length Killing vector fields on the geometry of Riemannian manifolds has been studied by several authors
from different points of view, e.g.,~\cite{N-2021,D-B-2021}.
The curvature of $f$-contact and $f$-{K}-contact manifolds was studied in \cite{CFF-1990,Di-T-2006,Di-T-2009}.
Observe that there are no Einstein metrics on $f$-{K}-contact manifolds.

\smallskip

In \cite{rst-43,Rov-splitting,RWo-2} the study of ``weak" contact structures on a smooth $(2n+s)$-dimensi\-onal manifold
(i.e., the complex structure on the characteristic distribution is replaced with a nonsingular skew-symmetric tensor)
was initiated. This generalizes a metric $f$-structure and its satellites, and allows a new look at the classical theory and new applications.
In~\cite{rst-43}, we proved that the ${\cal S}$-structure is rigid, i.e., our weak ${\cal S}$-structure is the ${\cal S}$-structure,
and a weak metric $f$-structure with parallel tensor $f$ is a weak~${\cal C}$-structure.
In~\cite{Rov-splitting}, we obtained a topological obstruction (including the Adams number) to the existence of weak $f$-K-contact manifolds.
In \cite{RWo-2}, we retracted weak structures with positive partial Ricci curvature onto the subspace of classical structures of the same type.

The article continues our study \cite{rst-43,Rov-splitting,RWo-2} of the geometry of weak $f$-contact manifolds.
Following the approach in \cite{r-23} for $s=1$, we study the sectional and Ricci curvature in the ${\xi}$-directions of
a weak $f$-contact structure and its particular case -- a weak $f$-{K}-contact structure.
The~charac\-teristic distribution of a {weak} $f$-contact manifold defines a totally geodesic foliation with flat leaves.
Our goal is to show that the weak structures can be a useful for studying unit Killing vector fields,
totally geodesic foliations and $\xi$-sectional curvature on Riemannian manifolds,
and that some results on $f$-{K}-contact manifolds can be extended to the case of weak $f$-{K}-contact manifolds.
For example, we answer the questions of when a weak $f$-{K}-contact manifold
(i)~carries a generalized Ricci soliton structure or just a gradient quasi Einstein (e.g., an Einstein)~metric, and
(ii)~carries an $\eta$-Ricci soliton structure or just an $\eta$-Einstein metric.
The proofs use the properties of new tensors, as well as the constructions required in the classical~case.


The article is organized as follows.
{Section~\ref{sec:01}}, following the introductory {Section~\ref{sec:00}}, presents basics of weak metric $f$-manifolds.
{Section~\ref{sec:03a}} presents basics of weak $f$-contact and weak $f$-K-contact manifolds
and shows that these manifolds are endowed with totally geodesic foliations.
We show that $\ker f$ of a weak $f$-contact manifold defines a $\mathfrak{g}$-foliation with an abelian Lie algebra.
In~{Section~\ref{sec:02}}, we characterize (Theorem~\ref{T-3.1}) weak $f$-{K}-contact manifolds among all weak $f$-contact manifolds by the property ${f}=-\nabla{\xi_i}\ (i\le s)$ known for $f$-{K}-contact manifolds, and find sufficient conditions (Theorem~\ref{prop2.1b}) under which when a Riemannian mani\-fold endowed with a set of unit Killing vector fields is a weak $f$-{K}-contact manifold.
In~{Section~\ref{sec:03}}, for a weak $f$-{K}-contact manifold, we calculate (in Proposition~\ref{thm6.2D}) the Ricci curvature in the ${\xi}$-directions, prove that the $\xi$-sectional curvature is positive (Corollary~\ref{C-5.1}).
This allows us to deform a weak $f$-K-contact structure to the $f$-K-contact structure (Theorem~\ref{T-PRF-g2b}).
We~show that for $s>1$, an Einstein weak $f$-{K}-contact manifold is Ricci flat
and find (Theorem~\ref{T-4.1}) sufficient conditions for a weak $f$-{K}-contact manifold with parallel Ricci tensor to be a Ricci flat manifold,
or, for $s=1$, an Einstein manifold.
In~{Section~\ref{sec:04}}, we find (Theorems~\ref{T-5.1} and \ref{T-5.2}) sufficient conditions for a weak $f$-{K}-contact manifold with a generalized gradient Ricci soliton to be a quasi Einstein or Ricci flat manifold, or, for $s=1$, an Einstein manifold.
In~{Section~\ref{sec:05}}, we define an $\eta$-Ricci soliton structure and find sufficient conditions for
a compact weak $f$-{\rm K}-contact manifold with such a structure of constant scalar curvature to be~$\eta$-Einstein (Theorem~\ref{T-7.1}).

\section{Preliminaries}
\label{sec:01}

Here, we survey the basics of a weak metric $f$-structure (see \cite{rst-43,RWo-2})
as a higher dimensional analog of a weak almost contact structure (see \cite{RovP-arxiv}).
We define weak $f$-contact and weak $f$-{K}-contact structures and show that these structures are endowed with totally geodesic~foliations.

A~\textit{framed weak $f$-structure} on a smooth manifold $M^{2n+s}$ is a set $({f},Q,{\xi_i},{\eta^i})$,
where ${f}$ is a $(1,1)$-tensor, $Q$ is a nonsingular $(1,1)$-tensor, ${\xi_i}\ (1\le i\le s)$ are characteristic vector fields and ${\eta^i}$ are dual 1-forms, satisfying
\begin{equation}\label{2.1}
 {f}^2 = -Q + \sum\nolimits_{\,i}{\eta^i}\otimes {\xi_i},\quad
 {\eta^i}({\xi_j})=\delta^i_j,\quad
 Q\,{\xi_i} = {\xi_i}.
\end{equation}
The forms ${\eta^i}$ determine a $2n$-dimensional contact distribution ${\cal D}:=\bigcap_{\,i}\ker{\eta^i}$.
Assume that ${\cal D}$ is ${f}$-invariant
(that is true for framed $f$-structure \cite{b1970,YK-1985}, where $Q={\rm id}_{\,TM}$),
\begin{equation}\label{2.1-D}
 {f} X\in{\cal D}\quad (X\in {\cal D}) .
\end{equation}
By \eqref{2.1} and \eqref{2.1-D}, the distribution ${\cal D}$ is invariant for $Q$: $Q({\cal D})={\cal D}$.
Note that $f^3+fQ=0$.
For a framed weak $f$-structure on a manifold $M$, the tensor ${f}$ has rank $2n$, thus ${\cal D}=f(TM)$, and
\[
 {f}\,{\xi_i}=0,\quad {\eta^i}\circ{f}=0,\quad \eta^i\circ Q=\eta^i,\quad [Q,\,{f}]=0 .
\]
A framed weak $f$-structure $({f},Q,{\xi_i},{\eta^i})$ is called {\it normal} if the following tensor is zero:
\begin{align*}
 N^{\,(1)}(X,Y) = [{f},{f}](X,Y) + 2\sum\nolimits_{\,i} d{\eta^i}(X,Y)\,{\xi_i},\quad X,Y\in\mathfrak{X}_M ,
\end{align*}
where
the Nijenhuis torsion $[{f},{f}]$ of ${f}$ is given~by
\begin{align}\label{2.5}
 [{f},{f}](X,Y) = {f}^2 [X,Y] + [{f} X, {f} Y] - {f}[{f} X,Y] - {f}[X,{f} Y],\quad X,Y\in\mathfrak{X}_M,
\end{align}
and the exterior derivative of a differential form ${\eta^i}$ is given by
\begin{equation}\label{3.3A}
 d\eta^i(X,Y) = \frac12\,\{X({\eta^i}(Y)) - Y({\eta^i}(X)) - {\eta^i}([X,Y])\},\quad X,Y\in\mathfrak{X}_M.
\end{equation}
Recall the following formulas with the Lie derivative $\pounds_{Z}$ in the $Z$-direction:
\begin{align}\label{3.3B}
\nonumber
 (\pounds_{Z}{f})X &  = [Z, {f} X] - {f} [Z, X]\quad (X,Y\in\mathfrak{X}_M),\\
 (\pounds_{{\xi_i}}\,g)(X,Y) & = {\xi_i}(g(X,Y)) -g([{\xi_i},X],Y) - g(X,[{\xi_i},Y]).
\end{align}
The following tensors $N^{\,(2)}_i, N^{\,(3)}_i$ and $N^{\,(4)}_{ij}$ are well known in the classical theory, see \cite{b1970,YK-1985}:
\begin{align*}
 N^{\,(2)}_i(X,Y) &= (\pounds_{{f} X}\,{\eta^i})(Y) - (\pounds_{{f} Y}\,{\eta^i})(X)
 \overset{\eqref{3.3A}}=2\,d{\eta^i}({f} X,Y) - 2\,d{\eta^i}({f} Y,X) ,  \\
 N^{\,(3)}_i(X) &= (\pounds_{{\xi_i}}\,{f})X
 = [{\xi_i}, {f} X] - {f} [{\xi_i}, X],\\
 N^{\,(4)}_{ij}(X) &= (\pounds_{{\xi_i}}\,{\eta^j})(X)
 = {\xi_i}({\eta^j}(X)) - {\eta^j}([{\xi_i}, X])
 \overset{\eqref{3.3A}}= 2\,d{\eta^j}({\xi_i}, X),
\end{align*}
where $\pounds_{Z}$ is the Lie derivative in the ${Z}$-direction.
For $s=1$, the above tensors reduce to the following tensors on (weak) almost contact manifolds:
\begin{align*}
 N^{\,(2)}(X,Y) = (\pounds_{f\,X}\,\eta)Y - (\pounds_{f\,Y}\,\eta)X, \quad
 N^{\,(3)} = \pounds_{\xi}\,f,\quad
 N^{\,(4)} = \pounds_{\xi}\,\eta .
\end{align*}

\begin{remark}[see \cite{r-23}]\rm
Let $M^{2n+s}(f,Q,\xi_i,\eta^i)$ be a framed weak $f$-manifold.
Consider the product manifold $\bar M = M^{2n+s}\times\mathbb{R}^s$,
where $\mathbb{R}^s$ is a Euclidean space with a basis $\partial_1,\ldots,\partial_s$,
and define tensors $\bar f$ and $\bar Q$ on $\bar M$ putting
\begin{align*}
 \bar f(X,\, \sum\nolimits_{\,i} a^i\partial_i) = (fX-\sum\nolimits_{\,i} a^i\xi_i,\, \sum\nolimits_{\,j} \eta^j(X)\partial_j),\quad
 \bar Q(X,\, \sum\nolimits_{\,i} a^i\partial_i) = (QX,\, \sum\nolimits_{\,i} a^i\partial_i) ,
\end{align*}
where $a_i\in C^\infty(M)$.
Hence, $\bar f(X,0)=(fX,0)$, $\bar Q(X,0)=(QX,0)$ for $X\in{\cal D}$,
$\bar f(\xi_i,0)=(0,\partial_i)$, $\bar Q(\xi_i,0)=(\xi_i,0)$ and
$\bar f(0,\partial_i)=(-\xi_i,0)$, $\bar Q(0,\partial_i)=(0,\partial_i)$.
Then it is easy to verify that $\bar f^{\,2}=-\bar Q$.
 The~tensors $N^{\,(1)},N^{\,(2)}_i, N^{\,(3)}_i, N^{\,(4)}_{ij}$ appear when we derive the integrability condition $[\bar f, \bar f]=0$
of $\bar f$ and express the normality condition $N^{\,(1)}=0$ of a framed weak $f$-structure on~$M$.
\end{remark}

 If there is a Riemannian metric $g$ on $M$ such that
\begin{align}\label{2.2}
 g({f} X,{f} Y)= g(X,Q\,Y) -\sum\nolimits_{\,i}{\eta^i}(X)\,{\eta^i}(Y),\quad X,Y\in\mathfrak{X}_M,
\end{align}
then $({f},Q,{\xi_i},{\eta^i},g)$ is called a {\it weak metric $f$-structure} on $M$, and $g$
is called a \textit{compatible} metric.
Note that $X = X^\top + X^\bot$, where
$X^\top = \sum_{\,i}\eta^i(X)\,\xi_i$ is the projection of the vector $X\in TM$ on~$\ker f$.
A~framed weak $f$-manifold $M({f},Q,{\xi_i},{\eta^i})$ endowed with a compatible Riemannian metric $g$
is called a \textit{metric weak $f$-manifold} and is denoted by $M({f},Q,{\xi_i},{\eta^i},g)$.
Putting $Y={\xi_i}$ in \eqref{2.2} and using $Q\,{\xi_i}={\xi_i}$, we get, as in the classical theory, ${\eta^i}(X)=g(X,{\xi_i})$;
thus $g(\xi_i,\xi_j)=\delta_{ij}$.
In~particular, ${\xi_i}$ is $g$-orthogonal to ${\cal D}$ for any compatible metric $g$.
Thus, the orthogonal complement ${\cal D}^\bot$ of ${\cal D}$ is the characteristic (or, Reeb) distribution $\ker f$.
Some sufficient conditions for the existence of a compatible metric on a weak almost contact manifold are given in~\cite{RWo-2}.

Using the Levi-Civita connection $\nabla$ of $g$, one can rewrite \eqref{2.5} as
\begin{align}\label{4.NN}
 [{f},{f}](X,Y) = ({f}\nabla_Y{f} - \nabla_{{f} Y}{f}) X - ({f}\nabla_X{f} - \nabla_{{f} X}{f}) Y .
\end{align}
For a weak metric $f$-structure, the distribution ${\cal D}$ is
non-integrable (has no integral $2n$-dimensio\-nal submanifolds)
since for any nonzero $X\in{\cal D}$ we get
\[
 g([X, {f} X], {\xi_i})= 2\,d{\eta^i}({f} X,X) = g({f} X,{f} X)>0.
\]
For a weak metric $f$-structure, the tensor ${f}$ is skew-symmetric and $Q$ is positive definite and self-adjoint, see \cite{rst-43,RWo-2}:
\begin{equation*}
 g({f} X, Y) = -g(X, {f} Y),\quad
 g(QX,Y)=g(X,QY),\quad X,Y\in\mathfrak{X}_M.
\end{equation*}
The {fundamental $2$-form} $\Phi$ on $M({f},Q,\xi_i,\eta^i,g)$ is defined by
 $\Phi(X,Y)=g(X,{f} Y),\ X,Y\in\mathfrak{X}_M$.
Recall the co-boundary formula for exterior derivative $d$ on a $2$-form $\Phi$,
\begin{align}\label{E-3.3}
 d\Phi(X,Y,Z) &= \frac{1}{3}\,\big\{ X\,\Phi(Y,Z) + Y\,\Phi(Z,X) + Z\,\Phi(X,Y) \notag\\
 &-\Phi([X,Y],Z) - \Phi([Z,X],Y) - \Phi([Y,Z],X)\big\}.
\end{align}

\begin{remark}\rm
A differential $k$-\textit{form} on a manifold $M$ is a skew-symmetric tensor field
$\omega$ of  type $(0, k)$. According the conventions of e.g., \cite{KN-69}, the formula
\begin{eqnarray}\label{eq:extdiff}
\nonumber
 & d\omega ({X}_1, \ldots , {X}_{k+1}) = \frac1{k+1}\sum\nolimits_{\,i=1}^{k+1} (-1)^{i+1} {X}_i(\omega({X}_1, \ldots , \hat{{X}}_i\ldots, {X}_{k+1}))\\
 & +\sum\nolimits_{\,i<j}(-1)^{i+j}\omega ([{X}_i, {X}_j], {X}_1, \ldots,\hat{{X}}_i,\ldots,\hat{{X}}_j, \ldots, {X}_{k+1}),
\end{eqnarray}
where ${X}_i\in\mathfrak{X}_M$ and $\,\hat{\ }\,$ denotes the
operator of omission, defines a $(k+1)$-form $d\omega$ called the \textit{exterior differential} of $\omega$.
Note that \eqref{3.3A} and \eqref{E-3.3} correspond to \eqref{eq:extdiff} with $k=1$ and $k=2$.
\end{remark}


\section{Weak $f$-contact and weak $f$-{K}-contact structures}
\label{sec:03a}

Here, we define ``weak" structures that generalize $f$-contact and $f$-{K}-contact structures.

\begin{definition}\rm
A \textit{weak $f$-contact structure} is a weak metric $f$-structure satisfying
\begin{align}\label{2.3}
 \Phi=d{\eta^1}=\ldots =d{\eta^s} ,
\end{align}
thus, $d\Phi=0$.
A weak $f$-contact structure is called a \textit{weak $f$-{\rm K}-contact structure} if
all characteristic vector fields ${\xi_i}$ are Killing,~i.e.
\begin{eqnarray}\label{E-nabla}
\nonumber
 & (\pounds_{{\xi_i}}\,g)(X,Y):= {\xi_i}(g(X,Y)) -g([{\xi_i},X],Y) - g(X,[{\xi_i},Y]) \\
 & =g(\nabla_X {\xi_i}, Y) +g(\nabla_Y {\xi_i}, X)=0 .
\end{eqnarray}
\end{definition}

For a weak $f$-contact structure and all $i$ and $j$, by \eqref{2.3}, we get
\[
 d\eta^j(\xi_i,X) = \Phi(\xi_i,X)=0.
\]
For a weak $f$-contact structure, the distribution ${\cal D}$ is
non-integrable (has no integral $2n$-dimensio\-nal submanifolds)
since for any nonzero $X\in{\cal D}$ we get
 $g([X, {f} X], {\xi_i})= 2\,d{\eta^i}({f} X,X) = g({f} X,{f} X)>0$.
The relationships between the different classes of weak manifolds (considered in this article) can be summarizes in the diagram
well known for classical structures:
\[
 \left|\!\!\!
   \begin{array}{c}
   \textrm{framed\ weak} \\
   f\textrm{-manifold} \\
   \end{array}
 \!\!\!\right|
\overset{\textrm{metric}}\longrightarrow
\left|\!\!\!
   \begin{array}{c}
  \textrm{metric\ weak} \\
    f\textrm{-manifold}\\
   \end{array}
 \!\!\!\right|
\overset{\Phi=d\eta^i}\longrightarrow
  \left|\!\!\!
   \begin{array}{c}
  \textrm{weak}  \\
    f\textrm{-contact} \\
   \end{array}
 \!\!\!\right|
\overset{\xi_i\,\textrm{-Killing}}\longrightarrow
  \left|\!\!\!
   \begin{array}{c}
  \textrm{weak}  \\
   f\textrm{-K-contact}\\
   \end{array}
 \!\!\!\right|
 \overset{N^{\,(1)}=0}
\longrightarrow
 \left|\!\!\!
   \begin{array}{c}
 \textrm{weak} \\
  {\cal S}\textrm{-manifold} \\
   \end{array}
 \!\!\!\right|.
\]
For $s=1$, we get the following diagram:
\[
 \left|\!\!\!
   \begin{array}{c}
  \textrm{weak\ almost} \\
  \textrm{contact} \\
   \end{array}
 \!\!\!\right|
\overset{\textrm{metric}}\longrightarrow
\left|\!\!
   \begin{array}{c}
  \textrm{weak\ almost} \\
  \textrm{contact\ metric}\\
   \end{array}
 \!\!\right|
\overset{\Phi=d\eta}\longrightarrow
  \left|\!\!
   \begin{array}{c}
  \textrm{weak\ contact} \\
  \textrm{metric} \\
   \end{array}
 \!\!\right|
\overset{\xi\,\textrm{-Killing}}\longrightarrow
  \left|\!\!\!
   \begin{array}{c}
  \textrm{weak} \\
     \textrm{K-contact} \\
   \end{array}
 \!\!\!\right|
\overset{N^{\,(1)}=0}
\longrightarrow
 \left|\!\!\!
   \begin{array}{c}
 \textrm{weak} \\
 \textrm{Sasakian} \\
   \end{array}
 \!\!\!\right|.
\]

\begin{remark}\rm
In \cite{DIP-2001,rst-43,Di-T-2006}, the $f$-contact structure is called the almost ${\cal S}$-structure,
and the normal $f$-contact structure is called the \textit{${\cal S}$-structure}.
A weak metric $f$-structure with the properties $d\Phi=0$ and $N^{\,(1)}=0$ is called a weak ${\cal K}$-structure;
in this case, $\xi_1,\ldots,\xi_s$ are Killing vector fields, see \cite{rst-43};
thus, we get a weak $f$-{K}-contact structure.
\end{remark}

\begin{proposition}[see Theorem~2.2 in \cite{rst-43}]\label{thm6.2}
For a weak $f$-contact structure, the tensors $N^{\,(2)}_i$ and $N^{\,(4)}_{ij}$ vanish;
moreover, $N^{\,(3)}_i$ vanishes if and only if $\,\xi_i$ is a Killing vector field.
In~particular, for a weak $f$-K-contact structure, the tensors $N^{\,(2)}_i$, $N^{\,(3)}_i$ and $N^{\,(4)}_{ij}$ vanish.
\end{proposition}

A ``small" (1,1)-tensor $\widetilde{Q} = Q - {\rm id}_{\,TM}$ measures the difference between a weak $f$-contact structure and a contact one,
and $\widetilde Q=0$ means the classical $f$-contact geometry.
We have
\[
 [\widetilde{Q},{f}]=0,\quad \widetilde{Q}\,{\xi_i}=0,\quad \eta^i\circ\widetilde{Q}=0.
\]

\begin{proposition}[see Corollary~2.1 in \cite{rst-43}]\label{lem6.1}
For a weak $f$-contact structure we get
\begin{align}\label{3.1A}
 2 g((\nabla_{X}{f})Y,Z) = g(N^{\,(1)}(Y,Z),{f} X) {+} 2 g(fX,fY)\bar\eta(Z) {-} 2g(fX,fZ)\bar\eta(Y)
  + N^{\,(5)}(X,Y,Z),
\end{align}
where $\bar\eta=\sum\nolimits_{\,i}\eta^i$, and a skew-symmetric with respect to $Y$ and $Z$ tensor $N^{\,(5)}(X,Y,Z)$ is given~by
\begin{align*}
 N^{\,(5)}(X,Y,Z) &= ({f} Z)\,(g(X, \widetilde QY)) -({f} Y)\,(g(X, \widetilde QZ))
 +g([X, {f} Z], \widetilde QY) - g([X,{f} Y], \widetilde QZ) \\
 & +\, g([Y,{f} Z] -[Z, {f} Y] - {f}[Y,Z],\ \widetilde Q X).
\end{align*}
Taking $X=\xi_i$ in \eqref{3.1A}, we get
\begin{align}\label{3.1AA}
 2\,g((\nabla_{\xi_i}{f})Y,Z) &= N^{\,(5)}(\xi_i,Y,Z) ,\quad 1\le i\le s.
\end{align}
\end{proposition}

Recall that a distribution $\widetilde{\cal D}\subset TM$ is \textit{totally geodesic} if and only if
$\nabla_X Y+\nabla_Y X\in\widetilde{\cal D}$ for any vector fields $X,Y\in\widetilde{\cal D}$ --
this is the case when {any geodesic of $M$ that is tangent to $\widetilde{\cal D}$ at one point is tangent to $\widetilde{\cal D}$ at all its points}, e.g., \cite[Section~1.3.1]{Rov-Wa-2021}. An~integrable totally geodesic distribution determines a totally geodesic foliation.

Let $\mathfrak{g}$ be a Lie algebra of dimension $s$.
We say that a foliation $\mathcal F$ of dimension $s$ on a smooth connected manifold $M$  is a $\mathfrak{g}$-\textit{foliation} if there exist
complete vector fields $\xi_1,\ldots,\xi_s$ on $M$, which restricted to each leaf of $\mathcal F$ form a parallelism of this submanifold
which is Lie algebra isomorphic to $\mathfrak{g}$, see, for example, \cite{RWo-2}.
A Lie algebra $\mathfrak{g}$ is abelian if its bracket is identically zero.
We show that $\ker f$ of a weak $f$-contact manifold defines a $\mathfrak{g}$-foliation with an abelian Lie algebra $\mathfrak{g}$.

\begin{corollary}
For a weak $f$-contact structure, the distribution $\ker f$ is integrable and defines a
$\mathfrak{g}$-foliation with totally geodesic flat leaves.
\end{corollary}

\begin{proof}
Recall that $[\xi_i,\xi_j]=\nabla_{\xi_i}\,\xi_j -\nabla_{\xi_j}\,\xi_i$.
Taking $Y=\xi_j$ in \eqref{3.1AA}, we get
\begin{align}\label{2-xi}
 \nabla_{\xi_i}\,\xi_j = 0,\qquad 1\le i,j\le s.
\end{align}
By \eqref{2-xi}, $\ker f$ of a weak $f$-contact structure defines
a $\mathfrak{g}$-foliation with an abelian Lie algebra (that is $[\xi_i,\xi_j]=0$)
with totally geodesic (that is $\nabla_{\xi_i}\,\xi_j = 0$) flat leaves (that is $R_{\xi_i,\xi_j}\,\xi_k=0$).
\end{proof}

\begin{remark}\rm
Only one new tensor $N^{(5)}$ (vanishing at $\widetilde Q=0$), which supplements the sequence of well-known tensors $N^{\,(i)},\ i=1,2,3,4$,
is needed to study the weak $f$-contact structure.
For~particular values of the tensor $N^{\,(5)}$ we get
\begin{align}\label{KK}
\nonumber
 N^{\,(5)}(X,\xi_i,Z) & = -N^{\,(5)}(X, Z, \xi_i) = g( N^{\,(3)}_i(Z),\, \widetilde Q X),\\
\nonumber
 N^{\,(5)}(\xi_i,Y,Z) &= g([\xi_i, {f} Z], \widetilde QY) -g([\xi_i,{f} Y], \widetilde QZ),\\
 N^{\,(5)}(\xi_i,\xi_j,Y) &= N^{\,(5)}(\xi_i,Y,\xi_j)=0.
\end{align}
\end{remark}

\begin{definition}[see \cite{rst-43}]\rm
Two framed weak $f$-structures $({f},Q,{\xi_i},{\eta^i})$ and $({f}',Q',{\xi_i},{\eta^i})$ on a smooth manifold $M$ are said to be \textit{homothetic}
if the following is valid for some real $\lambda>0$:
\begin{align}\label{Tran'}
  {f} = \sqrt\lambda\ {f}', \quad
  Q\,|_{\,{\mathcal D}}=\lambda\,Q'|_{\,\mathcal D} .
\end{align}
Two weak metric $f$-structures $({f},Q,{\xi_i},{\eta^i},g)$ and $({f}',Q',{\xi_i},{\eta^i},g')$ on $M$
are said to be \textit{homothetic} if they satisfy conditions \eqref{Tran'} and
\begin{align}\label{Tran2'}
 g|_{\,{\mathcal D}} = \lambda^{-\frac12}\,g'|_{\,{\mathcal D}},\quad
 g({\xi_i},\,\cdot) = {g}'({\xi_i},\,\cdot) .
\end{align}
\end{definition}

\begin{lemma}[see \cite{rst-43}]\label{P-22}
Let $({f},Q,{\xi_i},{\eta^i})$ be a framed weak $f$-structure such that
 $Q\,|_{\,{\mathcal D}}=\lambda\,{\rm id}_{\mathcal D}$
for some real $\lambda>0$. Then the following is true:

\noindent
$\bullet$ $({f}', {\xi_i}, {\eta^i})$ is a framed $f$-contact structure, where ${f}'$ is given by \eqref{Tran'}$_1$.

\noindent
$\bullet$ If $({f},Q,{\xi_i},{\eta^i},g)$ is a weak $f$-contact structure
and \eqref{Tran'}$_1$ and \eqref{Tran2'} are valid,
then $({f}',{\xi_i},{\eta^i},{g}')$ is an $f$-contact structure.
\end{lemma}

\section{Killing vector fields of unit length}
\label{sec:02}

Here, we characterize weak $f$-{K}-contact manifolds among all weak $f$-contact manifolds
and find conditions under which a Riemannian manifold endowed with a set of unit Killing vector fields is a weak $f$-{K}-contact manifold.
The tensor $N^{\,(3)}_i$ is important for $f$-contact manifolds.
Therefore, we define for a {weak} $f$-contact manifold the tensor fields frame $h=(h_1,\ldots,h_s)$,~where
\begin{align*}
 h_i=({1}/{2})\, N^{\,(3)}_i = ({1}/{2})\,\pounds_{\xi_i}{f} .
\end{align*}
By Proposition~\ref{thm6.2}, $h_i=0$ if and only if $\xi_i$ is a Killing vector field.
First, we calculate
\begin{align}\label{4.2}
 (\pounds_{\xi_i}{f})X \overset{\eqref{3.3B}} = \nabla_{\xi_i}({f} X) - \nabla_{{f} X}\,\xi_i - {f}(\nabla_{\xi_i}X - \nabla_{X}\,\xi_i)
 = (\nabla_{\xi_i}{f})X - \nabla_{{f} X}\,\xi_i + {f}\nabla_X\,\xi_i.
\end{align}
Taking $X=\xi_i$ in \eqref{4.2} and using $\nabla_{\xi_i}\,\xi_j=0$, see \eqref{2-xi}, and
 $(\nabla_{\xi_i}{f})\,\xi_j =0$,
see \eqref{3.1AA} with $Y=\xi_j$, we get the equality $h_i\,\xi_j = 0$.
From \eqref{4.2}, using $g(\nabla_X\,\xi_i,\xi_j)=0$ and $g((\nabla_{\xi_i}{f})\,\xi_j,Z)=\frac12 N^{\,(5)}(\xi_i,\xi_j,Z)=0$, see \cite{rst-43},
we conclude that $\eta_j\circ h_i=0$.

For an $f$-contact structure, the linear operator $h_i$ is self-adjoint, trace-free and anti-commutes with $f$,
i.e., $h_i{f}+{f}\, h_i=0$, see \cite{DIP-2001}.
We generalize this for a weak $f$-contact structure.

\begin{proposition}[see \cite{rst-43,Rov-splitting}]\label{P-4.1}\rm
For a weak $f$-contact structure $(f,Q,\xi_i,\eta^i,g)$, the tensor $h_i$ and its conjugate tensor $h_i^*$ satisfy
\begin{eqnarray*}
 g((h_i-h_i^*)X, Y) &=& ({1}/{2})\,N^{\,(5)}(\xi_i, X, Y),\quad X,Y\in TM,\\
 h_i{f}+{f}\, h_i &=& -(1/2)\,\pounds_{\xi_i}\widetilde{Q} ,\\
 h_i{Q}-{Q}\,h_i &=& (1/2)[\,f,\, \pounds_{\xi_i}\widetilde{Q}\,] .
\end{eqnarray*}
\end{proposition}

The following corollary generalizes the known property of $f$-contact manifolds.

\begin{corollary}
 Let a weak $f$-contact manifold satisfy $\pounds_{\xi_i}\widetilde{Q} =0$, then
 $\tr h_i=0$.
\end{corollary}

\begin{proof}
If $h_i X=\lambda X$, then using $h_i f =- f h_i$ (by the assumptions and Proposition~\ref{P-4.1}), we get $h_i f X=-\lambda f X$.
Thus, if $\lambda$ is an eigenvalue of $h_i$, then $-\lambda$ is also an eigenvalue of $h_i$; hence, $\tr h_i=0$.
\end{proof}

\begin{lemma}
On a weak $f$-{K}-contact manifold $M^{2n+s}(f,Q,\xi_i,\eta^i,g)$, we get $N^{\,(1)}({\xi_i},\,\cdot)=0$~and
\begin{eqnarray}
\label{E-31}
 N^{\,(5)}({\xi_i},\,\cdot\,,\,\cdot) & = & N^{\,(5)}(\,\cdot\,,\,{\xi_i},\,\cdot) =0,\\
 \label{E-31A}
 \pounds_{{\xi_i}}\,\widetilde{Q} &=& \nabla_{\xi_i}\,\widetilde{Q} = 0 ,\\
\label{E-30-phi}
 \nabla_{{\xi_i}}\,{f} &=& 0.
\end{eqnarray}
\end{lemma}

\begin{proof}
From \eqref{4.NN} with $Y=\xi_i$, since ${f}\,\xi_i=0$, we get
\begin{align}\label{4.NNxi}
 [{f},{f}](X,\xi_i)= {f}(\nabla_{\xi_i}{f})X +\nabla_{{f} X}\,\xi_i -{f}\,\nabla_{X}\,\xi_i, \quad X\in \mathfrak{X}_M .
\end{align}
By \eqref{4.NNxi} and $d{\eta^j}({\xi_i},\,\cdot) = \Phi(\xi_i,\,\cdot) = 0$ we have
\[
 N^{\,(1)}({\xi_i},X) = [{f},{f}](X,{\xi_i}) = {f}^2 [X,{\xi_i}] - {f}[{f} X,{\xi_i}] = {f} N^{\,(3)}_i(X)=0.
\]
By \eqref{KK}$_1$, using $N^{\,(3)}_i=0$ we get $N^{\,(5)}(\,\cdot\,, {\xi_i},\,\cdot) = N^{\,(5)}(\,\cdot\,, \cdot\,, \xi_i) =0$.
Then, by Proposition~\ref{P-4.1} with $h_i:=\frac12\,N^{\,(3)}_i=0$, we get $N^{\,(5)}({\xi_i},\,\cdot\,,\,\cdot) = 0$,
 $\pounds_{{\xi_i}}{Q} =0$ and
\begin{equation*}
 g(Q\,\nabla_{X}\,{\xi_i}, Z) = g({f} Z,QX) - (1/2)\,N^{\,(5)}(X,{\xi_i},{f} Z)  = g({f} Z,QX).
\end{equation*}
We use $[{f}, {Q}]=0$ and $Q=\widetilde{Q}+{\rm id}_{\,TM}$ to obtain $\nabla_{\xi_i}\,\widetilde{Q}=0$:
\[
  (\pounds_{{\xi_i}}\,{Q})X = [{\xi_i}, {Q}X] - {Q}[{\xi_i}, X] = (\nabla_{\xi_i}{Q})X +[{f}, {Q}]X
 =(\nabla_{\xi_i}\,{Q})X =(\nabla_{\xi_i}\,\widetilde{Q})X.
\]
This completes the proof of \eqref{E-31} and \eqref{E-31A}.
Next, from \eqref{3.1AA} and \eqref{E-31} we get \eqref{E-30-phi}.
\end{proof}

In the next theorem, we characterize weak $f$-{K}-contact manifolds among all weak $f$-contact manifolds
by the following well known property of $f$-{K}-contact manifolds, see \cite{b1970,Goertsches-2}:
\begin{equation}\label{E-30}
 \nabla\,{\xi_i} = -{f},\quad 1\le i\le s .
\end{equation}

\begin{theorem}\label{T-3.1}
A weak $f$-contact manifold is weak $f$-{K}-contact if and only if \eqref{E-30} holds.
\end{theorem}

\begin{proof}
Let a weak $f$-contact manifold satisfy \eqref{E-30}. By skew-symmetry of ${f}$, we get
\[
 (\pounds_{{\xi_i}}\,g)(X,Y)=g(\nabla_X\,{\xi_i}, Y)+g(\nabla_Y\,{\xi_i}, X)=-g({f} X, Y)-g({f} Y, X)=0,
\]
thus, ${\xi_i}$ are Killing vector~fields.

Conversely, let our manifold be weak $f$-{K}-contact.
By \eqref{3.1A} with $Y={\xi_i}$, using \eqref{E-31}, we get
 $g((\nabla_{X}\,{f})\,{\xi_i},Z) =  g( {f}^2 X, Z)$.
Hence, $(\nabla_{X}\,{f})\,{\xi_i} = {f}^2 X$.
From this and
\[
 0=\nabla_{X}\,({f}\,{\xi_i})=(\nabla_{X}\,{f})\,{\xi_i}+{f}\nabla_{X}\,{\xi_i},
\]
we obtain ${f}(\nabla_{X}\,{\xi_i}+{f} X)=0$.
Since $\nabla_{{X}}\,\xi_i+f {X}\in{\cal D}$ and $f$ is invertible when restricted on ${\cal D}$,
we get $\nabla_{{X}}\,\xi_i= -f {X}$ that proves \eqref{E-30}.
\end{proof}

\begin{corollary}[see \cite{r-23}]
A weak contact manifold $M^{2n+1}(f,Q,\xi,\eta,g)$ is weak {K}-contact if and only if $\nabla\,{\xi} = -{f}$ holds.
\end{corollary}

Denote by $R_{{X},{Y}}Z=(\nabla_X\nabla_Y -\nabla_Y\nabla_X -\nabla_{[X,Y]})Z$ the curvature tensor.
The mapping $R_{\,{\xi_i}}: X \mapsto R_{X,\,{{\xi_i}}}\,{{\xi_i}}\ (\xi\in\ker f,\ \|\xi\|=1)$ will be called the \textit{Jacobi operator} in the ${\xi_i}$-direction, e.g., \cite{Rov-Wa-2021}.
For a weak $f$-contact manifold, by Proposition~\ref{lem6.1}, we get $R_{\,{\xi_i}}(X)\,\in{\cal D}$.
Recall that a Riemannian manifold with a unit Killing vector field and the property $R_{\xi}(X)=X\ (X\bot\,{\xi})$ is a {K}-contact manifold,
e.g., \cite[p.~268]{YK-1985}.
We generalize this result and \cite[Theorem~2]{r-23} in the following.

\begin{theorem}\label{prop2.1b}
A Riemannian manifold $(M^{2n+s},g)$ with orthonormal Killing vector fields ${\xi_i}\ (i=1,\ldots, s)$ such that
$d\,\eta^1=\ldots=d\,\eta^s$
$($where $\eta^i$ is the 1-form dual to $\xi_i)$
and the Jacobi operators $R_{\,{\xi_i}}\ (i\le s)$ are positive definite on the distribution ${\cal D}= \bigcap_{\,i}\ker\eta^i$
is a weak $f$-{K}-contact manifold with the following structural tensors:
\[
 {f} = -\nabla\,{\xi_i},\quad
 Q X = R_{\xi_i}(X)\quad
 (X\in{\cal D},\ \ i=1,\ldots, s).
\]
\end{theorem}

\begin{proof}
Put ${f} = -P\nabla\,{\xi_i}$ and $\Phi(X,Y)=g(X,{f} Y)$, where $P:TM\to {\cal D}$ is the orthoprojector,  that does no depend on $i$.
Since ${\xi_i}$ are Killing vector fields, we obtain the property \eqref{2.3}:
\[
 d{\eta^i}(X,Y) = (1/2)\,(g(\nabla_X\,{\xi_i}, Y) - g(\nabla_Y\,{\xi_i}, X)) = - g(\nabla_Y\,{\xi_i}, X) = g(X,{f} Y).
\]
Set $Q X = R_{\,{\xi_i}}(X)$ for some $i$ and all $X\in{\cal D}$.
Since ${\xi_i}$ is a unit Killing vector field, we get $\nabla_{\xi_i}\,{\xi_i}=0$ and
$\nabla_X\nabla_Y\,{\xi_i} - \nabla_{\nabla_X\,Y}\,{\xi_i} = R_{\,X,{\xi_i}}\,Y$, see \cite{YK-1985}.
Thus, ${f}\,{\xi_i}=0$ and
\[
 {f}^2 X = \nabla_{\nabla_X\,{\xi_i}}\,{\xi_i} = R_{\,{\xi_i},X}\,{\xi_i} =-R_{\xi_i}(X)= -QX,\quad X\in{\cal D}.
\]
By the conditions, the Jacobi operator $R_{\xi_i}$ does no depend on $i$ and the tensor $Q$ is positive definite on ${\cal D}$.
Thus, $f$ restricted to $\cal D$ has constant rank $2\,n$.
Put $Q\,{\xi_i} = {\xi_i}\ (i=1,\ldots,s)$. Therefore, \eqref{2.1} and \eqref{2.2} are valid, that completes the proof.
\end{proof}

\begin{corollary}[see \cite{r-23}]
A Riemannian manifold $(M^{2n+1},g)$ with a unit Killing vector field $\xi$ such that
$R_{\,\xi}$ is positive definite on ${\cal D}=\{X\in TM: g(X,\xi)=0\}$ is a weak K-contact manifold
with the structural tensors $\eta(X) = g(X, \xi)$, $f = -\nabla\,\xi$ and $Q X = R_{\,\xi}(X)$ for $X\in{\cal D}$.
\end{corollary}

If a plane contains unit vectors: $\xi\in\ker f$ and $X\in{\cal D}$, then its sectional curvature is called $\xi$-\textit{sectional curvature}.
The $\xi$-sectional curvature of an $f$-contact manifold is an example of \textit{mixed sectional curvature} of an almost product manifold,
for example,~\cite{Rov-Wa-2021}.
It is well known that the $\xi$-sectional curvature of an $f$-K-contact manifold~is constant equal to 1.

\begin{proposition}
A weak $f$-{K}-contact structure $({f},Q,{\xi_i},{\eta^i},g)$ with constant mixed sectional curvature,
$K({\xi_i},X)=\lambda>0$ for some $\lambda=const\in\mathbb{R}$ and all ${X}\in{\cal D}$,
is homothetic to an $f$-{K}-contact structure $({f}',{\xi_i},{\eta^i},g')$ after the transformation \eqref{Tran'}--\eqref{Tran2'}
with $Q'={\rm id}$.
\end{proposition}

\begin{proof}
Note that $K({\xi_i},X)=\lambda\ (X\in{\cal D})$ if and only if $R_{\,{\xi_i}}(X) = \lambda\,X\ (X\in{\cal D})$.
By $QX=R_{\,{\xi_i}}(X)\ (X\in{\cal D})$, see Theorem~\ref{prop2.1b}, we get $Q X = \lambda X\ (X\in{\cal D})$.
By~Lemma~\ref{P-22}(ii), $({f}',{\xi_i},{\eta^i},{g}')$ is an $f$-contact structure.
Using \eqref{E-nabla}, we get $(\pounds_{{\xi_i}}\,g')({\xi_i},\,\cdot)=0$ and
 $(\pounds_{{\xi_i}}\,g')(X,Y)=\lambda (\pounds_{{\xi_i}}\,g)(X,Y)$ for $X,Y\in{\cal D}$.
By $\pounds_{{\xi_i}}\,g=0$, we get $\pounds_{{\xi_i}}\,g'=0$; thus $({f}',{\xi_i},{\eta^i},g')$ is an $f$-{K}-contact structure.
\end{proof}

\begin{example}\rm
By Theorem~\ref{prop2.1b}, we can search for examples of weak $f$-K-contact (not $f$-K-contact) manifolds
among Riemannian manifolds of positive sectional curvature that admit a set of $s\ge1$ orthonormal Killing vector fields.
Such example with $s=1$ is given in \cite{r-23} using idea of \cite[p.~5]{D-B-2021}.
Indeed, let $M$ be a convex hypersurface (ellipsoid) with induced metric $g$ of~$\mathbb{R}^{2n+2}$,
\[
 M = \Big\{(u_1,\ldots,u_{2n+2})\in\mathbb{R}^{2n+2}: \sum\nolimits_{\,i=1}^{n+1} u_i^2 + a\sum\nolimits_{\,i=n+2}^{2n+1} u_i^2 = 1\Big\},
\]
where $0<a=const\ne1$ and $n\ge1$. The sectional curvature of $(M,g)$ is positive. It~follows that
\[
 \xi = (-u_2, u_1, \ldots , -u_{n+1}, u_{n}, -\sqrt a\,u_{n+3}, \sqrt a\,u_{n+2}, \ldots , -\sqrt a\,u_{2n+2}, \sqrt a\,u_{2n+1})
\]
is a Killing vector field on $\mathbb{R}^{2n+2}$, whose restriction to $M$ has unit length.
Since $\xi$ is tangent~to~$M$, i.e., $M$ is invariant under the flow of $\xi$, $\xi$ is a unit Killing vector field on $(M,g)$.
\end{example}

\section{The curvature in the characteristic direction}
\label{sec:03}

Denote by $\Ric^\sharp$ the Ricci operator of $g$ associated with the Ricci tensor ${\rm Ric}$ and given by
\[
 {\rm Ric}({X},{Y})=g(\Ric^\sharp {X},{Y}),\quad {X},{Y}\in\mathfrak{X}_M.
\]
Since $K(\xi_i,\xi_j)=0$ for a weak $f$-contact manifold, the~Ricci curvature in the ${\xi_j}$-direction is given~by
 ${\rm Ric}({{\xi_j}},{{\xi_j}})=\sum\nolimits_{\,i=1}^{\,2n} g(R_{e_i,\,{{\xi_j}}}\,{{\xi_j}}, e_i)$,
where $(e_i)$ is any local orthonormal basis of~${\cal D}$.
In the next proposition, we generalize some particular properties of $f$-K-contact manifolds for weak $f$-K-contact manifolds.

\begin{proposition}\label{thm6.2D}
For a weak $f$-{K}-contact manifold $M^{2n+s}(f,Q,\xi_i,\eta^i,g)$, we have for all~$i$:
\begin{eqnarray}
\label{E-R0}
 && R_{\,{\xi_i},\,X} = \nabla_X\,{f}\quad (X\in\mathfrak{X}_M), \\
\label{E-R1}
 && R_{X,\,{{\xi_i}}}\,{{\xi_j}} = -{f}^2 X\quad (X\in\mathfrak{X}_M), \\
\label{Eq-Ric-f}
 && \Ric^\sharp {\xi_i} = \Div f ,\\
\label{E-R1b}
 && {\rm Ric}({{\xi_i}},{{\xi_j}}) =\tr Q = 2\,n + \tr\widetilde Q .
\end{eqnarray}
\end{proposition}

\begin{proof}
Using \eqref{E-30}, we derive
\begin{eqnarray}\label{Eq-R}
\nonumber
 R_{Z,\,X}\,{{\xi_i}} \eq \nabla_Z(\nabla_X\,{{\xi_i}}) -\nabla_X(\nabla_Z\,{{\xi_i}}) -\nabla_{[Z,X]}\,{{\xi_i}}) \\
 \eq \nabla_X ({f} Z) - \nabla_Z ({f} X) + {f}([Z,X])
 = (\nabla_X\,{f})Z -(\nabla_Z\,{f})X.
\end{eqnarray}
Note that
 $(\nabla_X \Phi)(Y,Z)
 =-g((\nabla_X\,{f})Y, Z)$.
Using condition $d\Phi=0$ and \eqref{E-3.3}, we get
\begin{equation}\label{E-dPhi-three}
 (\nabla_X \Phi)(Y,Z)+(\nabla_Y \Phi)(Z,X)+(\nabla_Z \Phi)(X,Y)=0.
\end{equation}
From \eqref{Eq-R}, using \eqref{E-dPhi-three}
and skew-symmetry of $\Phi$, we get \eqref{E-R0}:
\begin{eqnarray*}
 g(R_{{{\xi_i}},\,X}\,Y , Z) \eq g(R_{Y,\,Z}\,{{\xi_i}}, X)
 \overset{\eqref{Eq-R}}= (\nabla_Z\,\Phi)(X,Y) + (\nabla_Y\,\Phi)(Z,X) \\
 \hspace*{-2.1mm}&\overset{\eqref{E-dPhi-three}}=&\hspace*{-2.1mm} -(\nabla_X\,\Phi)(Y,Z)
 = g((\nabla_X\,{f})Y,Z).
\end{eqnarray*}
By \eqref{E-R0} with $Y={\xi_j}$, using ${f}\,{\xi_j} =0$ and \eqref{E-30}, we get \eqref{E-R1}:
 $R_{{{\xi_i}}, X}\,{{\xi_j}} = (\nabla_X\,{f})\,{\xi_j} = -{f}\nabla_X\, {{\xi_j}} = {f}^2 X$.
For weak $f$-K-contact manifolds, by \eqref{E-R0} and \eqref{E-30-phi},
we get the equality \eqref{E-Ric-f}:
\begin{equation}\label{E-Ric-f}
 \Ric(\xi_j,X)=\sum\nolimits_{\,k=1}^{\,2n} g((\nabla_{e_k}\,f)\,X, e_k) = g(\Div f, X),
\end{equation}
where $(e_k)$ is any local orthonormal basis of~${\cal D}$.
By \eqref{E-Ric-f}, we get the equality \eqref{Eq-Ric-f}.

Using the definition ${\rm Ric}(\xi_i,\xi_j)=\tr(X \to R_{X,\xi_i}\,\xi_j)$ and $f\,\xi_i=0$, we get for all $i,j$,
\[
 \Ric({{\xi_i}},{{\xi_j}}) \overset{\eqref{E-R1}}= -\sum\nolimits_{\,k=1}^{\,2n} g({f}^2 e_k, e_k)
 \overset{\eqref{2.1}}= \sum\nolimits_{\,k=1}^{\,2n} g(Q e_k, e_k).
\]
By the above and the equality $\tr Q = 2\,n + \tr\widetilde Q$, \eqref{E-R1b} is true.
\end{proof}

\begin{corollary}\label{C-5.0}
There are no Einstein weak $f$-{K}-contact manifolds $M^{2n+s}(f,Q,\xi_i,\eta^i,g)$ with $s{>}1$.
\end{corollary}

\begin{proof}
 For a weak $f$-{K}-contact manifold with $s>1$ and vector field $\tilde\xi=(\xi_1+\xi_2)/\sqrt 2$, we~get
\begin{equation}\label{E-Einst}
  {\rm Ric}(\tilde\xi,\tilde\xi) = (1/2)\sum\nolimits_{i,j=1}^2{\rm Ric}(\xi_i,\xi_j) \overset{\eqref{E-R1b}}= 2\tr Q.
\end{equation}
If $(M,g)$ is an Einstein manifold, then for the unit vector field $\tilde\xi$ we get ${\rm Ric}(\tilde\xi,\tilde\xi)={\rm Ric}(\xi_i,\xi_i)=\tr Q$.
Comparing with \eqref{E-Einst} gives a contradiction: $\tr Q=0$.
\end{proof}

\begin{corollary}\label{C-5.1}
For a weak $f$-K-contact manifold $M^{2n+s}(f,Q,\xi_i,\eta^i,g)$, the $\xi$-sectional curvature is positive:
\begin{equation*}
 K(\xi_i,{X})=g(Q{X},{X})>0\quad ({X}\in{\cal D},\ \|{X}\|=1),
\end{equation*}
and for the Ricci curvature we get ${\rm Ric}({\xi}_i,{\xi}_j)>0$ for all $1\le i,j\le s$.
\end{corollary}

\begin{proof}
For any unit vectors $\xi\in\ker f,\ X\in{\cal D}$, from \eqref{E-R1}, we get $K(\xi,X)=g(fX,fX)>0$.
From \eqref{E-R1b}, using \eqref{2.1} and non-singularity of $f$ on ${\cal D}$, we get the following:
\[
 \Ric({{\xi}_i},{{\xi}_j}) = \tr Q = -\sum\nolimits_{k=1}^{2n} g(f^2 e_k,  e_k) = \sum\nolimits_{k=1}^{2n} g(f e_k, f e_k)>0,
\]
where $(e_k)$ is a local orthonormal frame of ${\cal D}$,
thus the second statement is valid.
\end{proof}

The next theorem generalizes \cite[Proposition 5.1 on p.~283]{YK-1985} and \cite[Theorem~4]{r-23} for $s=1$.

\begin{theorem}\label{T-4.1}
A weak $f$-{K}-contact manifold $M^{2n+s}(f,Q,\xi_i,\eta^i,g)$ satisfying $\tr Q=const$ and $(\nabla\Ric)({\xi_i},\,\cdot)=0$,
e.g., the Ricci tensor is parallel, is
an Einstein manifold and $s=1$.
\end{theorem}

\begin{proof}
Differentiating \eqref{E-R1b} and using \eqref{E-30} and the conditions, we have
\[
 0 = \nabla_Y\,({\rm Ric}({{\xi_i}},{{\xi_i}})) = (\nabla_Y\,{\rm Ric})({{\xi_i}},{{\xi_i}}) +2\,{\rm Ric}(\nabla_Y\,{{\xi_i}},{{\xi_i}})
 = -2\,{\rm Ric}({f} Y,{{\xi_i}}),
\]
hence ${\rm Ric}(Y,{{\xi_i}}) = (\tr Q)\,{\eta^i}(Y)$.
Differentiating this, then using
\[
 X({\eta^i}(Y))=g(\nabla_X{\xi_i}, Y)=-g({f} X,Y)+g(\nabla_X Y, {\xi_i})
\]
and assuming $\nabla_X Y=0$ at $x\in M$, gives
\begin{equation*}
  (\tr Q)\,g({f} Y, X) = \nabla_X\,({\rm Ric}(Y,{{\xi_i}}))
  = (\nabla_X\,{\rm Ric})(Y,{{\xi_i}}) +2\,{\rm Ric}(Y, \nabla_X\,{{\xi_i}}) = -2\,{\rm Ric}(Y, {f} X).
\end{equation*}
Thus, ${\rm Ric}(Y, {f} X) = (\tr Q)\,g(Y, {f} X)$.
Therefore,
 ${\rm Ric}(X,Y) = (\tr Q)\,g(X,Y)$
for any vector fields $X$ and $Y$ on $M$. By the above and \eqref{E-R1b}, $(M,g)$
is
an Einstein manifold and $s=1$.
\end{proof}

\begin{corollary}[see \cite{r-23}]
A weak K-contact manifold $M^{2n+1}(f,Q,\xi,\eta,g)$ with $(\nabla\Ric)(\xi,\,\cdot)=0$,
e.g., the Ricci tensor is parallel, and $\tr Q=const$ is an Einstein manifold of scalar curvature
$(2\,n + 1)\tr Q$.
\end{corollary}

\begin{remark}\rm
For $f$-contact manifolds, we have
\begin{equation}\label{Eq-div-f}
 \sum\nolimits_{\,i=1}^{\,2n}(\nabla_{e_i}\,f)\,e_i=2\,n\,\bar\xi,
\end{equation}
see~\cite[Proposition~2.6]{DIP-2001}.
For $f$-K-contact manifolds, \eqref{Eq-Ric-f} and \eqref{Eq-div-f} give $\Ric^\sharp(\xi_i)=2\,n\,\bar\xi\ (i=1,\ldots,s)$
and $\Ric(\xi_i,\xi_i)=2\,n$.
Can one generalize this and \eqref{Eq-div-f} for weak $f$-contact manifolds\,?
\end{remark}

Define the skew-symmetric operator ${T}^{\sharp}_\xi\ (\xi\in\ker f)$ by
 $g({T}^{\sharp}_\xi X,Y)=g({T}(X,Y),\xi)\ (X,Y\in{\cal D})$,
where $T(X,Y) =(1/2)(\nabla_X Y-\nabla_Y X)^\bot$ is the integrability tensor of ${\cal D}$.

The self-adjoint {partial Ricci curvature tensor}, see \cite{RWo-2}, is given by
\[
 \Ric^\bot(X)=\sum\nolimits_{\,i} ({R}_{\,X^\bot,\,\xi_i}\,\xi_i)^\bot,
\]
and for an $f$-K-contact manifold, we have $\Ric^\bot=s\,{\rm id}^\bot$.
By Corollary~\ref{C-5.1}, the tensor $\Ric^\bot$ of a~weak $f$-K-contact manifold is positive definite.
In \cite{RWo-2}, using the flow of metrics on a $\mathfrak{g}$-foliation, a deformation retraction of
a weak almost ${\cal S}$-structures with positive partial Ricci curvature onto the subspace of the aforementioned classical structures is constructed.
In the following theorem, using Corollary~\ref{C-5.1} and the method of \cite{RWo-2}, we show that a weak $f$-K-contact structure can be transformed into an $f$-K-contact structure.

\begin{theorem}\label{T-PRF-g2b}
Let $M^{2n+s}(f(0),Q_{\,0},\xi_i,\eta^i,g_0)$ be a weak $f$-K-contact manifold.
Then there exist metrics $g_t\ (t\in\mathbb{R})$
such that each $(f(t),Q_{\,t},\xi_i,\eta^i,g_t)$ is a weak $f$-K-contact structure on $M$ with structural tensors defined on ${\cal D}$ as
\begin{equation}\label{E-Q-phi}
 Q_{\,t}=(1/s)\Ric^\bot_t,\quad
 f(t)\,|_{\,{\cal D}}=T_{\xi_i}^\sharp(t).
\end{equation}
Moreover,~$g_t$  converges exponentially fast, as $t\to-\infty$, to a limit metric $\hat g$
with $\Ric^\bot_{\hat g}=s\,{\rm id}^\bot$ that gives an $f$-K-contact structure on $M$.
\end{theorem}

\begin{proof} The partial Ricci curvature tensor of a weak $f$-K-contact manifold is positive definite.
Thus, the result is similar to \cite[Theorem~1]{RWo-2}.
Consider the ``partial Ricci  flow" of metrics on $M$,
\begin{equation}\label{E-GF-Rmix-Phi}
 \partial_t\,g_t = -2\,(\Ric^\bot)^\flat_{g_t} +2\,s\,g^\bot_t ,
\end{equation}
 where the $(0,2)$-tensor $g^\bot$ is defined by $g^\bot(X,Y):=g(X^\bot,Y^\bot)$.
For $t=0$, we get on ${\cal D}$:
\[
 \Ric^\bot = -\sum\nolimits_{\,i} (T_{\xi_i}^\sharp)^2 = -\sum\nolimits_{\,i} \phi_i^2 = s\,Q ,
\]
and find $T_{\xi_i}^\sharp \Ric^\bot =\Ric^\bot T_{\xi_i}^\sharp$, see also \cite{RWo-2}. Thus,
 $\sum\nolimits_i T_{\xi_i}^\sharp \Ric^\bot T_{\xi_i}^\sharp = -(\Ric^\bot)^2$.
By~the above, we obtain the following ODE:
 ${\partial_t}\Ric^\bot = 4\Ric^\bot(\Ric^\bot -\,s\,{\rm id}^\bot)$.
By the theory of ODE's, there exists a unique solution $\Ric^\bot_t$ for $t\in\mathbb{R}$;
hence, a solution $g_t$ of \eqref{E-GF-Rmix-Phi} exists for $t\in\mathbb{R}$ and it is unique.
Observe that $(f(t),\xi_i,\eta^i,Q_t)$ with $f(t),Q_t$ given in \eqref{E-Q-phi}
is a weak $f$-K-contact structure on $(M,g_t)$.
By uniqueness of a solution, the flow \eqref{E-GF-Rmix-Phi} preserves the directions of eigenvectors of $\Ric^\bot$,
and each eigenvalue $\mu_i>0$ of $\Ric^\bot$ satisfies ODE
 $\dot\mu_i=4\mu_i\,(\mu_i-s)$
with $\mu_i(0)>0$.
This ODE has the following solution (function $\mu_i(t)$ on $M$ for any $t\in\mathbb{R}$):
\[
 \mu_i(t)= \frac{\mu_i(0)\,s}{\mu_i(0)+\exp(4\,s\,t)(s-\mu_i(0))}
\]
with $\lim\limits_{t\to-\infty}\mu_i(t)=s$.
Thus, $\lim\limits_{t\to-\infty}\Ric^\bot(t)=s\,{\rm id}^\bot$.
Let $\{e_i(t)\}$ be a $g_t$-orthonormal frame of ${\cal D}$ of eigenvectors associated with $\mu_i(t)$, we then get
 ${\partial_t} e_i = (\mu_i - s) e_i$.
Since $e_i(t)=z_i(t)\,e_i(0)$ with $z_i(0)=1$, then ${\partial_t} \log z_i(t) = \mu_i(t) - s$.
By the~above, $z_i(t) = (\mu_i(t)/\mu_i(0))^{1/4}$. Hence,
\[
 g_t(e_i(0),e_j(0))= z_i^{-1}(t)z_j^{-1}(t)\,g_t(e_i(t),e_j(t)) = \delta_{ij}(\mu_i(0)\mu_j(0)/(\mu_i(t)\mu_j(t)))^{1/4}.
\]
As $t\to-\infty$, $g_t$ converges to the metric $\hat g$ determined by
$\hat g(e_i(0),e_j(0))=\delta_{ij}\sqrt{\mu_i(0)/s}$.
\end{proof}

\begin{corollary}\label{C-PRF-g-1}
Let $M^{2n+1}(f(0),Q_{\,0},\xi,\eta,g_0)$ be a weak K-contact manifold.
Then there exist metrics $g_t\ (t\in\mathbb{R})$
such that each $(f(t),Q_{\,t},\xi,\eta,g_t)$ is a weak K-contact structure on $M$ with structural tensors
 $Q_{\,t}=\Ric^\bot_t$ and $f(t)\,|_{\,{\cal D}}=T_{\xi}^\sharp(t)$,
defined on ${\cal D}$.
Moreover, $g_t$  converges exponentially fast, as $t\to-\infty$, to a limit metric $\hat g$
with $\Ric^\bot_{\hat g}={\rm id}^\bot$ that gives a K-contact structure on~$M$.
\end{corollary}

\begin{remark}\rm
By Corollary~\ref{C-PRF-g-1} and results on K-contact manifolds, e.g., \cite[p.~40]{blair2010}, on a compact weak K-contact manifold $M^{\,2n+1}$, $\xi$ has at least $n+1$ closed orbits; and if $\xi$ has exactly $n+1$ closed orbits, then $M^{\,2n+1}$ is covered by a sphere.
\end{remark}

\section{Weak $f$-K-contact structure as a generalized Ricci soliton}
\label{sec:04}

Corollary~\ref{C-5.0} motivates to consider \textit{quasi Einstein manifolds}, introduced in \cite{CM-2000} by the condition
 $\Ric = a\,g + b\,\mu\otimes\mu$
for all vector fields $X,Y$,
where $a$ and $b\ne0$ are real scalars, and $\mu$ is a 1-form of unit norm, and later this structure has been studied by several geometers.
If $\mu$ is the differential of a function, then we get a \textit{gradient quasi Einstein manifold}.

The following generalization of \eqref{E-gg-r-e} was given in \cite{CZB-2022}:
\begin{equation}\label{E-gg-r-e2}
 {\rm Hess}_{\,\sigma_1} - c_2\Ric =  \lambda\,g -c_1 d\sigma_2\otimes d\sigma_2
\end{equation}
for some $\sigma_1,\sigma_2\in C^\infty(M)$ and real constants $c_1, c_2$ and $\lambda$. For $\sigma_1=\sigma_2$, \eqref{E-gg-r-e2} reduces to \eqref{E-gg-r-e}.

First, we formulate some lemmas.

\begin{lemma}[see Lemma 3.1 in \cite{G-D-2020} or Lemma~4 in \cite{r-23} for $s=1$]\label{L-5.1}
For a weak $f$-{K}-contact manifold $M^{2n+s}(f,Q,\xi_i,\eta^i,g)$ the following holds:
\[
 (\pounds_{{\xi_i}}(\pounds_{X}\,g))(Y,{\xi_i}) = g(X,Y) + g(\nabla_{\xi_i}\nabla_{\xi_i}\,X, Y) + Y g(\nabla_{\xi_i}\,X, {\xi_i})
\]
for any $i\le s$ and all smooth vector fields $X,Y$ with $Y$ orthogonal to $\ker f$.
\end{lemma}

\begin{proof} This uses the equalities $\nabla_{\xi_i}\,{\xi_i}=0$ and \eqref{E-R1}, and is the same as for \cite[Lemma~4]{r-23}.
\end{proof}

\begin{lemma}[see, for example, \cite{G-D-2020}]\label{L-5.2}
Let $(M; g)$ be a Riemannian manifold and $\sigma$ be a smooth function on $M$. Then the following holds for every vector fields $\xi,Y$ on $M:$
\[
 \pounds_{{\xi}}(d\sigma\otimes d\sigma)(Y,{\xi}) = Y({\xi}(\sigma))\,{\xi}(\sigma) + Y(\sigma)\,{\xi}({\xi}(\sigma)) .
\]
\end{lemma}

Recall that the Ricci curvature of any $f$-{K}-contact manifold satisfies the following condition:
\begin{equation}\label{E-K-Ric-X}
 \Ric(X,{\xi}) = 0\quad (X\in{\cal D},\ \xi\in\ker f).
\end{equation}

\begin{lemma}[see Lemma 3.3 in \cite{G-D-2020} or Lemma~6 in \cite{r-23} for $s=1$]\label{L-5.3}
Let a weak $f$-{K}-contact manifold $M^{2n+s}(f,Q,\xi_i,\eta^i,g)$ satisfy \eqref{E-K-Ric-X} and admit the generalized gradient Ricci soliton structure.
Then
\[
 \nabla_{\xi_i} \nabla\sigma = (\lambda + c_2\tr Q)\,{\xi_i} - c_1{\xi_i}(\sigma)\,\nabla\sigma .
\]
\end{lemma}

\begin{proof}
By \eqref{E-R1b} and \eqref{E-K-Ric-X} we get
\begin{equation}\label{E-Gh-3.4}
 \lambda\,{\eta^i}(Y) +c_2\Ric({\xi_i},Y) = (\lambda + c_2\tr Q)\,{\eta^i}(Y).
\end{equation}
Using \eqref{E-gg-r-e} and \eqref{E-Gh-3.4}, we get
\begin{equation}\label{E-Gh-3.5}
 {\rm Hess}_{\,\sigma}({\xi_i},Y) = -c_1{\xi_i}(\sigma)\,g(\nabla\sigma, Y) + (\lambda + c_2\tr Q)\,{\eta^i}(Y).
\end{equation}
Thus, \eqref{E-Gh-3.5} and the condition \eqref{E-gg-r-e} for the Hessian complete the proof.
\end{proof}

The following theorem generalizes \cite[Theorem~4]{r-23} and \cite[Theorem 3.1]{G-D-2020} with $s=1$.

\begin{theorem}\label{T-5.1}
Let a weak $f$-{K}-contact manifold $M^{2n+s}(f,Q,\xi_i,\eta^i,g)$ with $\tr Q=const$ satisfy the generalized gradient Ricci soliton equation
\eqref{E-gg-r-e} with $c_1(\lambda + c_2\tr Q)\ne -1$:
\begin{equation*}
 {\rm Hess}_{\,\sigma} - c_2\Ric = \lambda\,g -c_1 d\sigma\otimes d\sigma\,.
\end{equation*}
 Suppose that condition \eqref{E-K-Ric-X} is true.
Then $\sigma=const$. Furthermore, if $c_2\ne0$, then $(M,g)$ is an Einstein manifold and $s=1$.
\end{theorem}

\begin{proof}
Let $Y\bot \ker f$. Then by Lemma~\ref{L-5.1} with $X=\nabla f$, we obtain
\begin{equation}\label{E-G3.6}
 2\,(\pounds_{{\xi_i}}({\rm Hess}_{\,\sigma}))(Y,{\xi_i}) = Y(\sigma) + g(\nabla_{\xi_i}\nabla_{\xi_i}\nabla\sigma, Y) + Y g(\nabla_{\xi_i}\nabla\sigma, {\xi_i}).
\end{equation}
Using Lemma~\ref{L-5.3} in \eqref{E-G3.6} and the properties $\nabla_{\xi_i}\,{\xi_i}=0$ and $g({\xi_i},{\xi_i})=1$, yields
\begin{eqnarray}\label{E-G3.7}
\nonumber
 && 2\,(\pounds_{{\xi_i}}({\rm Hess}_{\,\sigma}))(Y,{\xi_i}) = Y(\sigma) + a\,g(\nabla_{\xi_i}\,{\xi_i}, Y) \\
\nonumber
 && -c_1 g(\nabla_{\xi_i} ({\xi_i}(\sigma)\nabla\sigma), Y) + a\,Y(g({\xi_i},{\xi_i})) - c_1 Y({\xi_i}(\sigma)^2) \\
 && = Y(\sigma) -c_1 g(\nabla_{\xi_i}({\xi_i}(\sigma)\nabla\sigma), Y) - c_1 Y({\xi_i}(\sigma)^2),
\end{eqnarray}
where $a:=\lambda + c_2\tr Q$.
Using Lemma~\ref{L-5.3} with $Y\in{\cal D}$, from \eqref{E-G3.7} it follows that
\begin{equation}\label{E-G3.8}
 2\,(\pounds_{{\xi_i}}({\rm Hess}_{\,\sigma}))(Y,{\xi_i}) = Y(\sigma) -c_1 {\xi_i}({\xi_i}(\sigma))\,Y(\sigma) +c_1^2{\xi_i}(\sigma)^2 Y(\sigma)
 - c_1 Y({\xi_i}(\sigma)^2).
\end{equation}
Since ${\xi_i}$ is a Killing vector field, thus $\pounds_{{\xi_i}}\,g=0$, this implies $\pounds_{{\xi_i}}\Ric=0$.
Using the above fact and applying the Lie derivative to equation \eqref{E-gg-r-e}, gives
\begin{equation}\label{E-G3.9}
 2\,(\pounds_{{\xi_i}}({\rm Hess}_{\,\sigma}))(Y,{\xi_i}) = -2\,c_1(\pounds_{{\xi_i}}(d\sigma\otimes d\sigma))(Y,{\xi_i}).
\end{equation}
Using \eqref{E-G3.8}, \eqref{E-G3.9} and Lemma~\ref{L-5.2}, we obtain
\begin{equation}\label{E-G3.10}
 Y(\sigma)\big(1 + c_1 {\xi_i}({\xi_i}(\sigma)) + c_1^2\,{\xi_i}(\sigma)^2 \big) = 0.
\end{equation}
By Lemma~\ref{L-5.3}, we get
\begin{equation}\label{E-G3.11}
 c_1 {\xi_i}({\xi_i}(\sigma)) = c_1\,{\xi_i}(g({\xi_i}, \nabla\sigma)) = c_1 g({\xi_i}, \nabla_{\xi_i}\nabla\sigma)
 = c_1 a - c_1^2\,{\xi_i}(\sigma)^2.
\end{equation}
Using \eqref{E-G3.10} in \eqref{E-G3.11}, we get $Y(\sigma)(c_1 a+1)=0$.
This implies $Y(\sigma)=0$ provided by $c_1 a+1 \ne 0$.
Hence, $\nabla\sigma$ is parallel to $\ker f$.
Taking the covariant derivative of $\nabla\sigma = \sum_{\,i}{\xi_i}(\sigma)\,{\xi_i}$ and using \eqref{E-30} and $\bar\xi=\sum_{\,i} \xi_i$, we obtain
\[
 g(\nabla_X\,\nabla\sigma, Y)= \sum\nolimits_{\,i} X({\xi_i}(\sigma))\,{\eta^i}(Y) -{\bar\xi}(\sigma)\,g({\sigma} X, Y),\quad X,Y\in\mathfrak{X}_M .
\]
From this, by symmetry of ${\rm Hess}_{\,\sigma}$, i.e. $g(\nabla_X\,\nabla\sigma, Y)=g(\nabla_Y\,\nabla\sigma, X)$,
we~get ${\bar\xi}(\sigma)\,g({\sigma} X, Y)=0$.
For $Y={\sigma} X$ with some $X\ne0$, since $g({\sigma} X, {\sigma} X)>0$, we get ${\bar\xi}(\sigma)=0$.
Repla\-cing $(\xi_i)$ with another orthonormal frame from $\ker f$ preserves the weak $f$-{K}-contact structure and allows reaching any direction ${\bar\xi}$ in $\ker f$. So $\nabla\sigma=0$, i.e. $\sigma=const$.
Thus, from \eqref{E-gg-r-e} and $c_2\ne0$ it follows the claim.
\end{proof}

\begin{corollary}[see \cite{r-23}]
Let a weak K-contact manifold $M(f,Q,\xi,\eta,g)$ with $\tr Q=const$ satisfy the generalized gradient Ricci soliton equation \eqref{E-gg-r-e} with $c_1(\lambda + c_2\tr Q)\ne -1$. Suppose that condition $\Ric({\xi},X) = 0\ (X\,\perp\,\xi)$
is true. Then $\sigma=const$. Furthermore, if $c_2\ne0$, then the manifold is an Einstein~one.
\end{corollary}

The following theorem contains 3 cases and generalizes (and uses) Theorem~\ref{T-5.1}.

\begin{theorem}\label{T-5.2}
Let a weak $f$-{K}-contact manifold $M^{2n+s}(f,Q,\xi_i,\eta^i,g)$ satisfy the generalized gradient Ricci soliton equation \eqref{E-gg-r-e2}
with $c_1 a \ne -1$, where $a=\lambda + c_2\tr Q$.
Suppose that conditions $\tr Q=const$ and \eqref{E-K-Ric-X} are true.
Then $\tilde\sigma=\sigma_1 +c_1 a \sigma_2$ is constant
and \begin{equation}\label{E-gg-r-e2b}
 -c_1 a\,{\rm Hess}_{\,\sigma_2} = -c_1\,d\sigma_2\otimes d\sigma_2 + c_2\Ric + \lambda\,g .
\end{equation}
Furthermore,

1. if $c_1 a\ne0$, then \eqref{E-gg-r-e2b} reduces to
 ${\rm Hess}_{\,\sigma_2} = \frac1a\,d\sigma_2\otimes d\sigma_2 - \frac{c_2}{c_1 a}\,\Ric - \frac\lambda{c_1 a}\,g$.
By Theorem~\ref{T-5.1}, if $c_1 a\ne-1$, then $\sigma_2=const$; moreover, if $c_2\ne0$, then $(M,g)$ is an Einstein manifold and $s=1$.

2. if $a=0$ and $c_1\ne0$, then \eqref{E-gg-r-e2b} reduces to
 $c_2\Ric -c_1\,d\sigma_2\otimes d\sigma_2 + \lambda\,g = 0$.
If $c_2\ne0$ and $\sigma_2\ne const$, then we get a gradient quasi Einstein manifold.

3. if $c_1=0$, then \eqref{E-gg-r-e2b} reduces to $c_2\Ric + \lambda\,g = 0$; moreover, if $c_2\ne0$, then $(M,g)$ is an Einstein manifold
and $s=1$.
\end{theorem}

\begin{proof} Similarly to Lemma~\ref{L-5.3}, we get
\begin{equation}\label{E-gg-r-e2c}
 \nabla_{\xi_i} \nabla\sigma_1 = a\,\xi_i - c_1\xi_i(\sigma_2)\,\nabla\sigma_2 .
\end{equation}
Using \eqref{E-gg-r-e2c} and Lemmas~\ref{L-5.1} and \ref{L-5.2}, and slightly modifying the proof of Theorem~\ref{T-5.1},
we find that the vector field $\nabla\tilde\sigma$ is parallel to $\ker f$, where $\tilde\sigma=\sigma_1 +c_1 a \sigma_2$.
As~in the proof of Theorem~\ref{T-5.1}, we get $d\tilde\sigma=0$, i.e.,
 $d\sigma_1 = -c_1 a\,d\sigma_2$.
Using this in \eqref{E-gg-r-e2}, we get \eqref{E-gg-r-e2b}.
Analyzing \eqref{E-gg-r-e2b}, we obtain the required 3 cases of the theorem.
\end{proof}

\section{$\eta$-Ricci solitons on compact weak $f$-K-contact manifolds}
\label{sec:05}

A weak metric $f$-manifold $M^{2n+s}({f},Q,{\xi_i},{\eta^i},g)$ is said to be \textit{$\eta$-Einstein},~if
\begin{equation}\label{Eq-2.10}
 \Ric = \alpha\,g +\beta\sum\nolimits_{\,i}\eta^i\otimes\eta^i +(\alpha+\beta)\sum\nolimits_{\,i<j}\eta^i\otimes\eta^j
\end{equation}
for some functions $\alpha,\beta\in C^\infty(M)$.

By an \textit{$\eta$-Ricci soliton} we mean here a weak $f$-contact manifold, whose Ricci tensor satisfies
\begin{equation}\label{Eq-1.1}
 (1/2)\,\pounds_V\,g + \Ric = \lambda\,g +\nu\sum\nolimits_{\,i}\eta^i\otimes\eta^i +(\lambda+\nu)\sum\nolimits_{\,i<j}\eta^i\otimes\eta^j
\end{equation}
for some functions $\lambda, \nu\in C^\infty(M)$ and some smooth vector field $V$ on $M$.

\begin{remark}\rm
For $s=1$, \eqref{Eq-2.10} and \eqref{Eq-1.1} give the well-known definitions on contact metric manifolds.
Let $s=1$ and $Q={\rm id}$. Then from \eqref{Eq-2.10} we get an $\eta$-Einstein structure $\Ric = \alpha\,g + \beta\,\eta\otimes\eta$,
e.g.,~\cite{G-2023};
and
from \eqref{Eq-1.1} we get an $\eta$-Ricci soliton $\frac12\,\pounds_V\,g + \Ric = \lambda\,g + \beta\,\eta\otimes\eta$ on an almost contact metric manifold, e.g.,~\cite{G-2023}.
%
Equation \eqref{Eq-1.1} gives a generalized $\eta$-Ricci soliton $\frac12\,\pounds_V\,g + \Ric = \lambda\,g +\nu\,\eta\otimes\eta$ on an almost contact metric manifold, e.g.,~\cite{G-2023};
if also $\nu=0$, then \eqref{Eq-1.1} gives a {Ricci almost soliton};
moreover, for $\lambda=const$ it is a Ricci soliton.
\end{remark}

For~weak $f$-K-manifolds and $V=\bar\xi$, the notion of $\eta$-Ricci soliton reduces to $\eta$-Einstein structure.
For this reason, in this section we study an $\eta$-Ricci soliton on a weak $f$-{\rm K}-contact manifold.

We~recall the generalized Pohozaev-Schoen identity due to Gover-Orsted \cite{G-O-2013}.

\begin{lemma}
 Let $(M^n, g)$ be a compact Riemannian manifold without boundary.
If $E$ is a divergence free symmetric {\rm (0,2)}-tensor and $V$ is a vector field on $M$, then
\begin{equation}\label{Eq-3.1}
 \int_{\,M} V(\tr_{\!g}\,E)\,{\rm d}\vol = \frac n2\int_{\,M} g(E^0, \pounds_V\,g)\,{\rm d}\vol,
\end{equation}
where $E^0=E-\frac1{n}\,(\tr_{\!g}\,E)\,g$ and ${\rm d}\vol$ is the volume form of $g$.
\end{lemma}

Denote by $r$ the scalar curvature of a Riemannian manifold $(M,g)$.
The constancy of $r$ plays an important role for the triviality of a compact gene\-ralized $\eta$-Ricci soliton,
see \cite
{G-2023}. We extend this result for compact weak $f$-{\rm K}-contact manifolds.

\begin{theorem}\label{T-7.1}
Let a compact weak $f$-{\rm K}-contact manifold $M^{2n+s}(f, \xi_i, \eta^i, Q, g)$ satisfy
$r=const$ and $\tr Q=const$. If $(g,V,\lambda,\nu)$ represents an $\eta$-Ricci soliton, then $M$ is $\eta$-Einstein.
\end{theorem}

\begin{proof}
We take the trace of \eqref{Eq-2.10} to get the scalar curvature $r = (2n + s)\,\alpha + s\,\beta$.
On the other hand, using \eqref{E-R1b} in \eqref{Eq-2.10}, we obtain $\alpha + \beta = \tr Q$. From these it follows that
\begin{equation}\label{Eq-3.2}
 \alpha = \frac{r-s\tr Q}{2\,n},\quad \beta = \frac{(2\,n+s)\tr Q - r}{2\,n}.
\end{equation}
Define a traceless symmetric (0,2)-tensor $E$ as
\begin{equation}\label{Eq-3.3}
 E = \Ric -\frac{r-s\tr Q}{2\,n}\,g
  +\frac{r-(2\,n+s)\tr Q}{2\,n}\sum\nolimits_{\,i}\eta^i\otimes\eta^i -(\tr Q)\sum\nolimits_{\,i<j}\eta^i\otimes\eta^j .
\end{equation}
We express \eqref{Eq-1.1}, using \eqref{Eq-3.3}, as
\begin{align}\label{Eq-3.5}
\nonumber
 (1/2)\,\pounds_V\,g + E & = \big(\lambda -\frac{r-s\tr Q}{2\,n}\big)\,g
 +\big(\nu +\frac{r-(2\,n+s)\tr Q}{2\,n}\big)\sum\nolimits_{\,i}\eta^i\otimes\eta^i\\
 & +(\lambda+\nu-\tr Q)\sum\nolimits_{\,i<j}\eta^i\otimes\eta^j.
\end{align}
Taking covariant derivative of \eqref{Eq-3.3} along $X$ and using \eqref{E-30} and $\tr Q=const$, gives
\begin{align}\label{Eq-3.6}
\nonumber
 & (\nabla_X\,E)(Y,Z) = (\nabla_X\Ric)(Y,Z) -\frac{X(r)}{2\,n}\,\big(g(Y,Z) -\sum\nolimits_{\,i}\eta^i(Y)\,\eta^i(Z)\big)\\
 & +\Big((s-1)\tr Q +\frac{(2\,n+s)\tr Q - r}{2\,n}\Big)\,\big(\bar\eta(Z)\,g(f X, Y) + \bar\eta(Y)\,g(f X, Z) \big).
\end{align}
Taking the trace of \eqref{Eq-3.6} over $X$ and $Z$, and since
$r=const$, $\Div_g\Ric = \frac12\,d\,r$ (twice-contracted 2nd Bianchi identity, \cite[p.~135]{KN-69}),
and $\tr f=0$, we get the equality $\Div_g E=0$:
\[
 (\Div_g E)(Y) = (\Div_g\Ric)(Y)
 +\frac{s\,(2\,n+1)\tr Q - r}{2\,n}\,\big(\bar\eta(e_i)\,g(f e_i, Y) + \bar\eta(Y)\,g(f e_i, e_i) \big) = 0.
\]
Since the tensor $E$ is traceless, we get $E^0=E$ and $V(\tr_{\!g} E)=0$.
Thus, applying \eqref{Eq-3.1}, we find
 $\int_{\,M} g(E, \pounds_V\,g)\,{\rm d}\vol =0$.
Applying \eqref{Eq-3.5} and \eqref{E-R1b} in the above integral formula,
and using equalities $g(E,g)= \tr_{\!g} E=0$ and $g(E, \eta^i\otimes\eta^j)= E(\xi_i, \xi_j)=0$,
we obtain
\begin{equation}\label{Eq-3.7}
 \int_{\,M} g(E, E)\,{\rm d}\vol =0.
\end{equation}
From \eqref{Eq-3.7} we get $E=0$.
Thus, our manifold is $\eta$-Einstein, see \eqref{Eq-2.10}, where $\alpha,\beta$ satisfy~\eqref{Eq-3.2}.
\end{proof}

The following consequence of Theorem~\ref{T-7.1} generalizes \cite[Theorem~3.1(i)]{G-2023}.

\begin{corollary}
Let a compact weak {\rm K}-contact manifold $M^{2n+1}(f, \xi, \eta, Q, g)$ satisfy $r=const$ and $\tr Q=const$.
If $(g,V,\lambda,\nu)$ represents a generalized $\eta$-Ricci soliton on $M$,
then $M$ is $\eta$-Einstein.
\end{corollary}

\section{Conclusion}

It is shown that a weak $f$-contact structure is a useful tool for studying unit Killing vector fields, totally geodesic foliations and mixed sectional curvature on Riemannian manifolds. Some results for $f$-contact and $f$-{K}-contact structures \cite{b1970,YK-1985} have been extended to certain weak structures. Conditions are found under which
(i)~a weak $f$-{K}-contact manifold with parallel Ricci tensor or with a generalized gradient Ricci soliton structure is a quasi Einstein or an Einstein manifold;
(ii)~a compact weak $f$-{K}-contact manifold with constant scalar curvature and an $\eta$-Ricci soliton structure is $\eta$-Einstein.
 Based on applications of the weak $f$-contact structure in contact geometry considered in the article,
we expect that this structure will also be fruitful in theoretical physics.

In conclusion, we pose several questions.
{Is the condition ``the $\xi$-sectional curvature is positive" sufficient for a weak $f$-contact manifold to be weak $f$-K-contact}?
{Does a weak $f$-contact manifold of dimension greater than $3$ have some positive $\xi$-sectional curvature}?
{Is a compact weak $f$-K-contact Einstein manifold an ${\cal S}$-manifold}?
{When a given weak $f$-K-contact manifold is a mapping torus (see \cite{Goertsches-2}) of a manifold of lower dimension}?
{When a weak $f$-contact manifold equipped with a Ricci-type soliton structure,
carries a canonical (for example, with constant sectional curvature or Einstein-type) metric}?
One could answer these questions by generalizing some deep results on $f$-contact manifolds for weak $f$-contact manifolds.


\end{document}